\documentclass[a4paper,10pt]{article}
\usepackage{amsmath}
\usepackage{amssymb}
\usepackage{amscd}
\usepackage{amsthm}

\newtheorem{theorem}{Theorem}
\newtheorem{lemma}[theorem]{Lemma}
\newtheorem{definition}[theorem]{Definition}
\newtheorem{remark}[theorem]{Remark}

\newtheorem{result}[theorem]{Result}

\newtheorem{corollary}[theorem]{Corollary}

\def\PG{\mathrm{PG}}
\def\C{\mathrm{C}_k}
\def\F{\mathbb{F}_q}
\def\V{\mathrm{V}}

\title{An empty interval in the spectrum of small weight codewords in the code from points and $k$-spaces of $\PG(n,q)$}
\author{M. Lavrauw \thanks{This author's research was supported by the Fund for Scientific Research Ð Flanders (FWO - Vlaanderen).} \and L. Storme \and P. Sziklai\thanks{This author was partially supported by OTKA T-049662, T-067867 and Bolyai
grants.} \and G. Van de Voorde \thanks{This author's research was supported by the Institute for the Promotion of Innovation through Science
and Technology in Flanders (IWT-Vlaanderen) and the Fund for Scientific Research Ð Flanders (FWO - Vlaanderen).} }

\begin{document}

\maketitle
\begin{abstract} Let  $\C(n,q)$ be the $p$-ary linear code defined by the incidence matrix of points and $k$-spaces in $\PG(n,q)$, $q=p^h$, $p$ prime, $h\geq 1$. In this paper, we show that there are no codewords of weight in the open interval $]\frac{q^{k+1}-1}{q-1},2q^k[$ in $\C(n,q)\setminus\mathrm{C}_{n-k}(n,q)^\bot$ which implies that there are no codewords with this weight in $\C(n,q)\setminus \C(n,q)^{\bot}$ if $k\geq n/2$. In particular, for the code $\mathrm{C}_{n-1}(n,q)$ of points and hyperplanes of $\PG(n,q)$, we exclude all codewords in $\mathrm{C}_{n-1}(n,q)$ with weight in the open interval $]\frac{q^n-1}{q-1},2q^{n-1}[$. This latter result implies a sharp bound on the weight of small weight codewords of $\mathrm{C}_{n-1}(n,q)$, a result which was previously only known for general dimension for $q$ prime and $q=p^2$, with $p$ prime, $p>11$, and in the case $n=2$, for $q=p^3$, $p\geq 7$ (\cite{chouinard2},\cite{fack},\cite{LSV1},\cite{LSV2}).
\end{abstract}

\section{Definitions}
Let $\PG(n,q)$ denote the $n$-dimensional projective space
over the finite field  $\F$ with $q$ elements, where $q=p^h$,  $p$ prime, $h\geq 1$, and let $\V(n+1,q)$ denote the underlying vector space. Let $\theta_n$ denote the number of points in $\PG(n,q)$, i.e., $\theta_n=(q^{n+1}-1)/(q-1)$.

We define the incidence matrix~$A = (a_{ij})$ 
of points and $k$-spaces in the projective space $\PG(n,q)$, $q=p^h$,  $p$~prime, $h\geq 1$, 
as the matrix whose rows are 
indexed by the $k$-spaces of $\PG(n,q)$ and whose columns are indexed
by the points of $\PG(n,q)$, and 
with entry
$$ 
a_{ij} = \left\{
\begin{array}{ll}
1 & \textrm{if point $j$ belongs to $k$-space $i$,}\\
0 & \textrm{otherwise.}
\end{array} 
\right.
$$
The $p$-ary linear code of points and $k$-spaces of $\PG(n,q)$, $q=p^h$, $p$ prime, $h\geq 1$,  is 
the $\mathbb{F}_p$-span of the rows of the  incidence matrix $A$. We denote this code by $\C(n,q)$. The {\em support} of a codeword $c$, denoted by $supp(c)$, is the set of all non-zero positions of $c$. The {\em weight} of $c$ is the number of non-zero positions of $c$ and is denoted by $wt(c)$. Often we identify the support of a codeword with the corresponding set of points of $\PG(n,q)$. We let $(c_1,c_2)$ denote the scalar product in $\mathbb{F}_p$ of two codewords $c_1, c_2$ of $\C(n,q)$. Furthermore, if $T$ is a set of points of $\PG(n,q)$, then the incidence vector of this set is also denoted by $T$.
The dual code $\C(n,q)^\bot$ is the set of all vectors orthogonal to all codewords of $\C(n,q)$, hence
$$\C(n,q)^\bot=\{ v\in V(\theta_n,p) || (v,c)= 0,\ \forall c\in \C(n,q)\}.$$
It is easy to see that $c\in\C(n,q)^{\bot}$ if and only if $(c,K) = 0$ for all $k$-spaces $K$ of $\PG(n, q)$.

\section{Previous results}
The $p$-ary linear code of points and lines of $\PG(2,q)$, $q=p^h$, $p$ prime, $h\geq 1$, is studied in \cite[Chapter 6]{AK}. In \cite[Proposition 5.7.3]{AK}, the codewords of minimum weight of the code of points and hyperplanes of $\PG(n,q)$, $q=p^h$, $p$ prime, $h\geq 1$, are determined. The first results on codewords of small weight in the $p$-ary linear code of points and lines in $\PG(2,p)$, $p$ prime, were proved by McGuire and Ward \cite{McGuire}, where they proved that there are no codewords of $\mathrm{C}_1(2,p)$, $p$ an odd prime, in the interval $[p+2,3(p+1)/2]$. This result was extended by Chouinard (see \cite{chouinard}, \cite{chouinard2}) where he proves the following result.
\begin{result}\label{ch}\cite{chouinard},\cite{chouinard2} In the $p$-ary linear code arising from $\PG(2,p)$, $p$ prime, there are no codewords with weight in the closed interval $[p+2,2p-1]$.
\end{result}
This result shows that there is a gap in the weight enumerator of the code $\mathrm{C}_1(2,p)$ of points and lines in $\PG(2,p)$, $p$ prime. In Corollary \ref{lgev}, Result \ref{ch} is extended to the code of points and $k$-spaces in $\PG(n,p)$, $p$ prime, $p>5$. 

In the case where $q$ is not a prime, we improve on results of \cite{LSV1} and \cite{LSV2}, where the authors exclude codewords of small weight in $\mathrm{C}_{n-1}(n,q)$, $q=p^h$, $p$ prime, $h\geq 1$, respectively $\C(n,q)\setminus\C(n,q)^{\bot}$, $q=p^h$, $p$ prime, $h\geq 1$, corresponding to linear small minimal blocking sets, which implied Result \ref{Re1} and Result \ref{Re}. For the definition of a blocking set, see the next section.
\begin{result}\label{Re1}\cite[Corollary 3]{LSV1} The only possible codewords $c$ of $\mathrm{C}_{n-1}(n,q)$, $q=p^h$, $p$ prime, $h\geq 1$, of weight in the open interval
$]\theta_{n-1},2q^{n-1}[$ are the scalar multiples of non-linear minimal blocking sets, intersecting every line in $1\pmod{p}$ points.\end{result}
\begin{result}\cite[Corollary 2]{LSV2} \label{Re} For  $k\geq n/2$, the only possible codewords $c$ of $\C(n,q)\setminus\C(n,q)^{\bot}$, $q=p^h$, $p$ prime, $h\geq 1$, of weight in the open interval $]\theta_k,2q^k[$ are scalar multiples of non-linear minimal $k$-blocking sets of $\PG(n,q)$, intersecting every line in $1\pmod{p}$ or zero points.
\end{result}
\begin{remark} It is believed (and conjectured, see \cite[Conjecture 3.1]{sziklai}) that all small minimal blocking sets are linear. If that conjecture is true, then Result \ref{Re1} eliminates all possible codewords of $\mathrm{C}_{n-1}(n,q)$, $q=p^h$, $p$ prime, $h\geq 1$, of weight in the open interval $]\theta_{n-1},2q^{n-1}[$, and Result \ref{Re} eliminates all codewords of $\C(n,q)\setminus\C(n,q)^{\bot}$, $q=p^h$, $p$ prime, $h\geq 1$, of weight in the open interval $]\theta_k,2q^k[$ if $k\geq n/2$.
\end{remark}
In this article, we avoid the obstacle of this non-solved conjecture and improve on Result \ref{Re1} and Result \ref{Re} by showing that there are no codewords in $\C(n,q)\setminus\mathrm{C}_{n-k}(n,q)^{\bot}$, $q=p^h$, $p$ prime, $p>5$, $h\geq 1$, in the open interval $]\theta_k,2q^k[$, which implies that there are no codewords in the open interval $]\theta_k,2q^k[$ in $\C(n,q)\setminus\C(n,q)^{\bot}$ if $k\geq n/2$. Using the results of \cite{LSV2}, we show that there are no codewords in $\C(n,q)$, $q=p^h$, $p$ prime, $h\geq 1$, $p> 7$, with weight in the open interval $]\theta_k,(12\theta_k+6)/7[$.

 In the case that $k=n-1$, we show that there are no codewords in $\mathrm{C}_{n-1}(n,q)$, $q=p^h$, $p$ prime, $h\geq 1$, in the open interval $]\theta_{n-1},2q^{n-1}[$. 
These bounds are sharp: codewords of minimum weight in $\mathrm{C}_{n-1}(n,q)$ have been characterized as scalar multiples of incidence vectors of hyperplanes (see \cite[Proposition 5.7.3]{AK}), and codewords of weight $2q^{n-1}$ can be obtained by taking the difference of the incidence vectors of two hyperplanes.

\section{Blocking sets}\label{Blo}
A {\em blocking set} of $\PG(n,q)$ is a
set $K$ of points such that each hyperplane of $\PG(n,q)$ contains at least one point
of $K$. A blocking set $K$ is called {\em trivial}
 if it contains a line of $\PG(n,q)$. 
These blocking sets are also called 
{\em $1$-blocking sets} in \cite{AB:80}. In general, a \emph{$k$-blocking
set} $K$ in $\PG(n,q)$ is a set of points  such that any $(n-k)$-dimensional
subspace intersects $K$. A $k$-blocking set $K$ is called {\em trivial} if there is a $k$-dimensional subspace contained in 
$K$. 
%The smallest non-trivial $k$-blocking sets are characterized as cones
%with a $(k-2)$-dimensional vertex $\pi_{k-2}$ and a non-trivial 1-blocking set
%of minimum cardinality in a plane, skew to $\pi_{k-2}$, of $\PG(n,q)$ as base curve \cite{AB:80,UH:98}.
If an $(n-k)$-dimensional space contains
exactly one point of a $k$-blocking set $K$ in $\PG(n,q)$, it is called a {\em tangent $(n-k)$-space} to $K$, and a point
$P$ of
$K$ is called {\em essential} when it belongs to a tangent $(n-k)$-space of $K$.
A $k$-blocking set $K$ is called {\em minimal} when no proper subset of $K$
is also a $k$-blocking set, i.e., when each point of $K$ is essential. A $k$-blocking set is called {\em small} if it contains less than $3(q^k+1)/2$ points.\\

 In order to define a {\em linear} $k$-blocking set, we introduce the notion of a Desarguesian spread.

By field reduction, the points of $\PG(n,q)$, $q=p^h$, $p$ prime, $h\geq 1$, correspond to $(h-1)$-dimensional subspaces of $\PG((n+1)h-1,p)$, since a point of $\PG(n,q)$ is a $1$-dimensional vector space over ${\mathbb F}_q$, and so an $h$-dimensional vector space over ${\mathbb F}_p$. In this way, we obtain a partition ${\mathcal D}$ of the point set of $\PG((n+1)h-1,p)$ by $(h-1)$-dimensional subspaces. In general, a partition of the point set of a projective space by subspaces of a given dimension $k$ is called a {\it spread}, or a {\it $k$-spread} if we want to specify the dimension. The spread we have obtained here is called a {\it Desarguesian spread}. Note that the Desarguesian spread satisfies the property that each subspace spanned by two spread elements is again partitioned by spread elements. 
%In fact, it can be shown   that 
%if the dimension of the ambient space is larger than twice the dimension plus one (i.e. 
%if $n\geq 2$, this property characterises a Desarguesian spread.

%From now on, we use the representation of minimal blocking sets in terms of spreads (see Lunardon \cite{L1},\cite{L2}). 
\begin{definition} Let $\mathcal{D}$ be a Desarguesian $(h-1)$-spread of $\PG((n+1)h-1,p)$ as defined above. If $U$ is a subset of $\PG((n+1)h-1,p)$, then we write $\mathcal{B}(U)=\lbrace R \in \mathcal{D}||U\cap R \neq \emptyset \rbrace$.\end{definition}

In analogy with the correspondence between the points of $\PG(n,q)$ and the elements of a Desarguesian spread $\mathcal D$ in $\PG((n+1)h-1,p)$, we obtain the correspondence between the lines of $\PG(n,q)$ and the $(2h-1)$-dimensional subspaces of $\PG((n+1)h-1,p)$ spanned by two elements of $\mathcal D$, and in general, we obtain the correspondence between the $(n-k)$-spaces of $\PG(n,q)$ and the $((n-k+1)h-1)$-dimensional subspaces of $\PG((n+1)h-1,p)$ spanned by $n-k+1$ elements of $\mathcal D$. With this in mind, it is clear that any $hk$-dimensional subspace $U$ of $\PG(h(n+1)-1,p)$ defines a $k$-blocking set ${\mathcal B}(U)$  in $\PG(n,q)$. A blocking set constructed in this way is called a {\it linear $k$-blocking set}.  Linear $k$-blocking sets were first introduced by Lunardon \cite[Section 5]{L1}, although there a different approach is used.
For more on the approach explained here, we refer to \cite[Chapter 1]{lavrauw2001}. 

\section{Results}
In \cite{sz}, Sz\H{o}nyi and Weiner proved the following result on small blocking sets.
\begin{result}\label{sz}\cite[Theorem 2.7]{sz} Let $B$ be a minimal blocking set of $\PG(n,q)$ with respect to $k$-dimensional subspaces, $q=p^h$, $p>2$ prime, $h\geq 1$, and assume that $\vert B \vert < 3(q^{n-k}+1)/2$. Then any subspace that intersects $B$, intersects it in $1 \pmod{p}$ points.
\end{result}

In \cite{LSV2}, Lavrauw et al. proved the following lemmas.
\begin{result}\label{st}The support of a codeword $c\in\C(n,q)$, $q=p^h$, $p$ prime, $h\geq 1$, with weight smaller than $2q^k$, for which $(c,S)\neq 0$ for some $(n-k)$-space $S$, is a minimal $k$-blocking set in $\PG(n,q)$. Moreover, $c$ is a scalar multiple of a certain incidence vector, and $supp(c)$ intersects every $(n-k)$-dimensional space in $1\pmod{p}$ points.
\end{result}
\begin{lemma} \label{R5} Let $c\in \mathrm{C}_k(n,q)$, $q=p^h$, $p$ prime, $h\geq 1$, then there exists a constant $a\in \mathbb{F}_p$ such that $(c,U)=a$, for all subspaces $U$ of dimension at least $n-k$.\end{lemma}

%\begin{result}\label{R4}Assume that $k\geq n/2$. A codeword $c$ of $\C(n,q)$ is in $\C(n,q)\cap \C(n,q)^\bot$ if and only if $(c,U)=0$ for all subspaces $U$ with $\dim(U)\geq n-k$.
%\end{result}
In the same way as is done by the authors in \cite[Theorem 19]{LSV2}, one can prove Lemma \ref{lem0}, which shows that all minimal $k$-blocking sets of size less than $2q^k$ and intersecting every $(n-k)$-space in $1 \pmod{p}$ points, are small.
\begin{lemma}\label{lem0} Let $B$ be a minimal  $k$-blocking set in
$\PG(n,q)$, $n\geq 2$,  $q=p^h$, $p$ prime, $p>5$, $h\geq 1$, 
 intersecting every $(n-k)$-dimensional space in $1 \pmod{p}$ points. If $|B|\in ]\theta_k,2q^k[$, then
\[ |B|< 
\frac{3(q^k-q^k/p)}{2}.\]
\end{lemma}
%Corollary \ref{ge} follows from Results \ref{st} and \ref{R4}, and Lemma \ref{lem0}.
%\begin{corollary}\label{ge} Let $k\geq n/2$. Let $c$ be a codeword of $\C(n,q)\setminus\C(n,q)^{\bot}$ of weight at most $2q^{k}-1$. Then $supp(c)$ is a small minimal $k$-blocking set in $\PG(n,q)$.
%\end{corollary}

\begin{lemma}\label{lem1} Let $B_1$ and $B_2$ be small minimal $(n-k)$-blocking sets in $\PG(n,q)$, $q=p^h$, $p$ prime, $h\geq 1$. Then $B_1-B_2\in \C(n,q)^\bot$.
\end{lemma}
\begin{proof}
It follows from Result \ref{sz} that $(B_i,\pi_k)=1$ for all $k$-spaces $\pi_k$, $i=1,2$. Hence $(B_1-B_2,\pi_k)=0$ for all $k$-spaces $\pi_k$. This implies that $B_1-B_2\in \C(n,q)^\bot$.\end{proof}

\begin{lemma} \label{lem2}  Let $c$ be a codeword of $\C(n,q)$, $q=p^h$, $p$ prime, $h\geq 1$, with weight smaller than $2q^k$, for which $(c,S)\neq 0$ for some $(n-k)$-space $S$, and let $B$ be a small minimal $(n-k)$-blocking set. Then $supp(c)$ intersects $B$ in $1 \pmod{p}$ points.
\end{lemma}
\begin{proof}
Let $c$ be a codeword of $\C(n,q)$ with weight smaller than $2q^k$, for which $(c,S)\neq 0$ for some $(n-k)$-space $S$. Lemma \ref{lem1} shows that $(c,B_1-B_2)=0=(c,B_1)-(c,B_2)$ for all small minimal $(n-k)$-blocking sets $B_1$ and $B_2$. Hence $(c,B)$, with $B$ a small minimal $(n-k)$-blocking set, is a constant. Result \ref{st} shows that $c$ is a codeword only taking values from $\lbrace 0,a \rbrace$, so $(c,B)=a(supp(c),B)$, hence $(supp(c),B)$ is a constant too. Let $B_1$ be an $(n-k)$-space, then Result \ref{st} shows that $(supp(c),B_1)=1$. Since $B_1$ is a small minimal $(n-k)$-blocking set, the number of intersection points of $supp(c)$ and $B$ is equal to 1 $\pmod{p}$ for any small minimal blocking set $B$.
\end{proof}
It follows from Lemma \ref{R5} that, for $c\in \C(n,q)$ and $S$ an $(n-k)$-space, $(c,S)$ is a constant. Hence, either $(c,S)\neq 0$ for all $(n-k)$-spaces $S$, or $(c,S)=0$ for all $(n-k)$-spaces $S$. In this latter case, $c\in\mathrm{C}_{n-k}(n,q)^{\bot}$.

\begin{theorem} \label{hoofd}There are no codewords in $\C(n,q)\setminus\mathrm{C}_{n-k}(n,q)^{\bot}$, $1\leq k\leq n-1$, $2\leq n$, with weight in the open interval $]\theta_k,2q^{k}[$, $q=p^h$, $p$ prime, $p>5$, $h\geq 1$.
\end{theorem}
\begin{proof} Let $Y$ be a linear small minimal $(n-k)$-blocking set in $\PG(n,q)$. As explained in Section \ref{Blo}, $Y$ corresponds to a set $\bar{Y}=\mathcal{B}(\pi)$ of $(h-1)$-dimensional spread elements intersecting a certain $(h(n-k))$-space $\pi$ in $\PG(h(n+1)-1,p)$. Let $c$ be a codeword of $\C(n,q)\setminus\mathrm{C}_{n-k}(n,q)^{\bot}$ with weight at most $2q^k-1$. Result \ref{st} and Lemma \ref{lem0} show that $supp(c)$ is a small minimal $k$-blocking set $B$. This blocking set $B$ corresponds to a set $\bar{B}$ of $\vert B \vert$ spread elements in $\PG(h(n+1)-1,p)$. Since $supp(c)$ and $Y$ intersect in $1 \pmod{p}$ points (see Lemma \ref{lem2}), $\bar{B}$ and $\bar{Y}$ intersect in $1\pmod{p}$ spread elements. Since all spread elements of $\bar{Y}$ intersect $\pi$, there are $1\pmod{p}$ spread elements of $\bar{B}$ that intersect $\pi$.

But this holds for any $(h(n-k))$-space $\pi'$ in $\PG(h(n+1)-1,p)$, since any $(h(n-k))$-space 
$\pi'$ corresponds to a linear small minimal $(n-k)$-blocking set $Y'$ in $\PG(n,q)$.

Let $\tilde{B}$ be the set of points contained in the spread elements of the set $\bar{B}$. Since a spread element that intersects a subspace of $\PG(h(n+1)-1,p)$ intersects it in $1 \pmod{p}$ points, $\tilde{B}$ intersects any $(h(n-k))$-space in $1 \pmod{p}$ points. Moreover, $\vert \tilde{B}\vert=\vert B \vert\cdot(p^h-1)/(p-1) \leq 3(p^{hk}-p^{hk-1})\cdot(p^h-1)/(2(p-1))< 3(p^{h(k+1)-1}+1)/2$ (see Lemma \ref{lem0}). This implies that $\tilde{B}$ is a small $(h(k+1)-1)$-blocking set in $\PG(h(n+1)-1,p)$.

 Moreover, $\tilde{B}$ is minimal. This can be proved in the following way. Let $R$ be a point of $\tilde{B}$. Since $B$ is a minimal $k$-blocking set in $\PG(n,q)$, there is a tangent $(n-k)$-space $S$ through the point $R'$ of $\PG(n,q)$ corresponding to the spread element $\mathcal{B}(R)$. Now $S$ corresponds to an $(h(n-k+1)-1)$-space $\pi'$ in $\PG(h(n+1)-1,p)$, such that $\mathcal{B}(R)$ is the only element of $\bar{B}$ in $\pi'$. This implies that through $R$, there is an $(h(n-k))$-space in $\pi'$ containing only the point $R$ of $\tilde{B}$. This shows that through every point of $\tilde{B}$, there is a tangent $(h(n-k))$-space, hence that $\tilde{B}$ is a minimal $(h(k+1)-1)$-blocking set.

Result \ref{sz} implies that $\tilde{B}$ intersects any subspace of $\PG(h(n+1)-1,p)$ in $1 \pmod{p}$ or zero points. This implies that a line is skew, tangent  or entirely contained in $\tilde{B}$, hence $\tilde{B}$ is a subspace of $\PG(h(n+1)-1,p)$, with at most $3(p^{h(k+1)-1}+1)/2$ points, intersecting every $(h(n-k))$-space. Moreover, it is the point set of a set of $\vert B \vert$ spread elements. Hence, $\bar{B}$ is the set of spread elements corresponding to a $k$-space in $\PG(n,q)$, so $supp(c)$ has size $\theta_k$.
\end{proof}

In \cite{LSV2}, Lavrauw et al. determined a lower bound on the weight of the code $\C(n,q)^{\bot}$.
\begin{result} \label{th8}The minimum weight of $\C(n,q)^\bot$, $q=p^h$, $p$ prime, $h\geq 1$, $2\leq n$, $1\leq k\leq n-1$, is at least $(12\theta_{n-k}+2)/7$ if $p=7$, and at least $(12\theta_{n-k}+6)/7$ if $p>7$.
\end{result}

\begin{theorem} \label{hoofd'}
For $q=p^h$, $p$ prime, $h\geq 1$, $2\leq n$, $1\leq k \leq n-1$, there are no codewords in $\C(n,q)$ with weight in the open interval $]\theta_k,(12\theta_{k}+2)/7[$ if $p=7$ and there are no codewords in $\C(n,q)$ with weight in the open interval $]\theta_k,(12\theta_{k}+6)/7[$ if $p>7$.
\end{theorem}
\begin{proof}
This follows immediately from Theorem \ref{hoofd} and Result \ref{th8}.
\end{proof}
In \cite{LSV2}, the authors proved the following result.

\begin{result} \cite[Lemma 3]{LSV2}\label{lem5} Assume that $k\geq n/2$. A codeword $c$ of $\C(n,q)$ is in $\C(n,q) \cap \C(n,q)^\bot$ if and only if $(c,U)=0$ for all subspaces $U$ with $\dim(U)\geq n-k$.
\end{result}
\begin{corollary}\label{gev3} If $k\geq n/2$, $\C(n,q)\setminus\mathrm{C}_{n-k}(n,q)^{\bot}=\C(n,q)\setminus\C(n,q)^{\bot}$, $q=p^h$, $p$ prime, $h\geq 1$.\end{corollary}
\begin{proof} It follows from Result \ref{lem5} that $\C(n,q)\cap \mathrm{C}_{n-k}(n,q)^{\bot}=\C(n,q) \cap \C(n,q)^{\bot}$ if $k\geq n/2$.
\end{proof}

In \cite{LSV1}, the authors proved the following result.
\begin{result} \cite[Theorem 5]{LSV1}\label{R}
The minimum weight of $\mathrm{C}_{n-1}(n,q)\cap \mathrm{C}_{n-1}(n,q)^\bot$ is equal to $2q^{n-1}$.
\end{result}
\begin{result}
\cite[Theorem 12]{LSV2}\label{R4}
The minimum weight of $\mathrm{C}_k(n,p)^{\bot}$, where $p$ is a prime, is equal to $2p^{n-k}$, and
the codewords of weight $2p^{n-k}$ are the scalar multiples of the difference of two $(n-k)$-spaces
intersecting in an $(n-k-1)$-space.
\end{result}

Theorem \ref{hoofd}, Corollary \ref{gev3}, and Result \ref{R} yield the following corollary, which gives a sharp empty interval on the size of small weight codewords of $\mathrm{C}_{n-1}(n,q)$, since $\theta_{n-1}$ is the weight of a codeword arising from the incidence vector of an $(n-1)$-space and $2q^{n-1}$ is the weight of a codeword arising from the difference of the incidence vectors of two $(n-1)$-spaces.

\begin{corollary} There are no codewords with weight in the open interval $]\theta_{n-1},2q^{n-1}[$ in the code $\mathrm{C}_{n-1}(n,q)$, $q=p^h$, $p$ prime, $h\geq 1$, $p>5$.
\end{corollary}

In the planar case, this yields the following corollary, which improves on the result of Chouinard mentioned in Result \ref{ch}.

\begin{corollary} There are no codewords with weight in the open interval $]q+1,2q[$ in the $p$-ary linear code of points and lines of $\PG(2,q)$, $q=p^h$, $p$ prime, $h\geq 1$, $p>5$.
\end{corollary}
In this case, the weight $q+1$ corresponds to the incidence vector of a line, and the weight $2q$ can be obtained by taking the difference of the incidence vectors of two different lines.

Theorem \ref{hoofd} and Result \ref{R4} yield the following corollary, extending the result of Chouinard mentioned in Result \ref{ch} to general dimension.

\begin{corollary} \label{lgev}There are no codewords with weight in the open interval $]\theta_{k},2q^{k}[$ in the code $\mathrm{C}_{k}(n,q)$, $n\geq 2$, $1\leq k \leq n-1$, $q$ prime, $q>5$.
\end{corollary}

Address of the authors: \\
\vspace{5pt}
\\
Michel Lavrauw, Leo Storme, Geertrui Van de Voorde:\\
Department of pure mathematics and computer algebra,\\
Ghent University\\
Krijgslaan 281-S22\\
9000 Ghent  (Belgium)\\
$\lbrace$ml,ls,gvdvoorde$\rbrace$@cage.ugent.be\\
http://cage.ugent.be/ $\sim$ $\lbrace$ml,ls,gvdvoorde$\rbrace$\\
\\
Peter Sziklai:\\
     Department of Computer Science,\\
    E\"otv\"os Lor\'and University\\
    P\'azm\'any P. s\'et\'any 1/C\\
    H-1117 Budapest (Hungary)\\
   sziklai@cs.elte.hu \\
     http://www.cs.elte.hu/$\sim$sziklai\\

\end{document}